\documentclass[11pt]{amsart}
\usepackage[colorlinks,citecolor=red, dvipdfm]{hyperref}
\usepackage{extarrows}

\newtheorem{theorem}{Theorem}[section]
\newtheorem{lemma}[theorem]{Lemma}

\theoremstyle{definition}

\newtheorem{question}[theorem]{Question}

\newtheorem{proposition}[theorem]{Proposition}
\newtheorem{corollary}[theorem]{Corollary}
\newtheorem{remark}[theorem]{Remark}

\theoremstyle{remark}

\newcommand{\be}{\begin{equation}}
\newcommand{\ee}{\end{equation}}

\numberwithin{equation}{section}

\begin{document}
\title{Spectral rigidity of complex projective spaces, revisited}
\author{Ping Li}
\address{School of Mathematical Sciences, Tongji University, Shanghai 200092, China}

\email{pingli@tongji.edu.cn\\
pinglimath@gmail.com}

\thanks{The author was partially supported by the National
Natural Science Foundation of China (Grant No. 11471247) and the
Fundamental Research Funds for the Central Universities.}




\begin{abstract}
A classical question in spectral geometry is, for each pair of nonnegative integers $(p,n)$ such that $p\leq 2n$, if the eigenvalues of the Laplacian on $p$-forms of a compact K\"{a}hler manifold are the same as those of $\mathbb{C}P^n$ equipped with the Fubini-Study metric, then whether or not this K\"{a}hler manifold is holomorphically isometric to $\mathbb{C}P^n$. For every positive even number $p$, we affirmatively solve this problem in all dimensions $n$ with at most two possible exceptions. We also clarify in this paper some gaps in previous literature concerned with this question, among which one is related to the volume estimate of Fano K\"{a}hler-Einstein manifolds.
\end{abstract}

\maketitle

\tableofcontents

\section{Introduction and main results}\label{section1}
Let $(M,g)$ be a compact $m$-dimensional Riemannian manifold, $\Omega^p(M)$ ($0\leq p\leq m$) the set of exterior $p$-forms on $M$ and $$d^{\ast}:~\Omega^p(M)\rightarrow\Omega^{p+1}(M)$$ the formal adjoint of the $d$-operator $$d:~\Omega^p(M)\rightarrow\Omega^{p+1}(M)$$ with respect to the Riemannian metric $g$. Here $\Omega^p(M)$ is understood to be zero if $p<0$ or $p>m$. Then for each $0\leq p\leq m$ we have the Laplacian
$$\Delta_p:=(d+d^{\ast})^2=dd^{\ast}+d^{\ast}d:~\Omega^p(M)
\longrightarrow\Omega^p(M),$$
which is a second-order self- adjoint elliptic operator. It is well-known from the spectral theory of self-adjoint operators that the eigenvalues of $\Delta_p$ form an infinite discrete sequence
$$0\leq\lambda_{1,p}\leq\lambda_{2,p}\leq\cdots\leq\lambda_{k,p}\leq
\cdots\uparrow+\infty$$
and each eigenvalue is repeated as many times as its multiplicity indicates. These $\lambda_{k,p}$ are called \emph{spectra} of the Laplacian with respect to $g$. For simplicity we denote by
$$\text{Spec}^p(M,g):=\big\{0\leq\lambda_{1,p}\leq\lambda_{2,p}\leq\cdots\leq\lambda_{k,p}\leq
\cdots\uparrow+\infty\big\}$$
and $\text{Spec}^p(M,g)$ is called the spectral set of $\Delta_p$. Poincar\'{e} duality and Hodge theory tell us that $\text{Spec}^p(M,g)=\text{Spec}^{m-p}(M,g)$ and $0\in\text{Spec}^p(M,g)$ if and only if the $p$-th Betti number $b_p(M)\neq0$ and its multiplicity is then $b_p(M)$.

It is an important problem to investigate how the geometry of $(M,g)$ can be reflected by its spectra $\{\lambda_{k,p}\}$. J. Milnor noted in \cite{Mi} that there exist two $16$-dimensional non-isometric Riemannian manifolds such that for each $p$ the spectrum sets $\text{Spec}^p(\cdot)$ are the same. This means in general the spectra $\{\lambda_{k,p}\}$ are not able to determine a manifold up to an isometry. Nevertheless, we may ask to what extent the spectra $\{\lambda_{k,p}\}$ determine the geometry of $(M,g)$. There have been some partial results towards this direction.  Notably are several results due to Patodi (\cite{Pa}), who, based on previous works of McKean-Singer and Berger (\cite{MS}, \cite{Be}), showed that whether or not $g$ is of flatness, has constant sectional curvature , or is an Einstein metric is completely determined by the spectra $\{\lambda_{k,p}\}$.

Although in general the spectral set $\{\lambda_{k,p}\}$ is not able to determine the whole geometry of $(M,g)$, we may still ask this question by putting more restrictions on the manifold $M$ and/or the metric $g$. To put this question into perspective, let us briefly recall some related background results. Suppose $(\mathbb{C}P^n, J_0)$ is the complex $n$-dimensional projective space with standard complex structure $J_0$. It is an important topic to characterize $(\mathbb{C}P^n, J_0)$ as a compact complex manifold via as little geometric/topological information as possible. To the author's best knowledge there are at least three classical characterizations of $(\mathbb{C}P^n, J_0)$, which are in terms of topology, geometry and curvature respectively. The first one is due to Hirzebruch-Kodaira and Yau (\cite{HK}, \cite{Yau}), which says that a K\"{a}hler manifold homeomorphic to $\mathbb{C}P^n$ must be biholomorphic to $(\mathbb{C}P^n, J_0)$. We refer the reader to \cite{To} and \cite{Li1} for a detailed proof and some technical improvements on this result. The second one is due to Kobayashi-Ochiai (\cite{KO}), which states that a Fano manifold whose Fano index is $n+1$ is biholomorphic to $(\mathbb{C}P^n, J_0)$ (more details on Fano index can be found in Section \ref{section2}). The third one was the famous Frankel conjecture solved by Mori and Siu-Yau independently (\cite{Mo}, \cite{SY}) saying that a compact K\"{a}hler manifold with positive holomorphic bisectional curvature is biholomorphic to $(\mathbb{C}P^n, J_0)$. Note that these three characterizations are deeply related to each other: the idea of the proof of \cite{KO} is inspired by some arguments in \cite{HK} while the result of \cite{KO} itself in turn plays an indispensable role in the proof of \cite{SY}.

Let $g_0$ be the Fubini-Study metric on $(\mathbb{C}P^n, J_0)$, which has constant holomorphic sectional curvature. With the above-mentioned background in mind, one of the most interesting problems on spectrum is the following question, which, to the author's best knowledge, should be first explicitly proposed in \cite{CV}.
\begin{question}\label{question}
Given a pair of nonnegative integers $(p,n)$ such that $p\leq 2n$. Suppose $(M,J,g)$ is a compact K\"{a}hler manifold with $\text{Spec}^p(M,g)=\text{Spec}^p(\mathbb{C}P^n, g_0)$. Is it true that $(M,J,g)$ is holomorphically isometric to $(\mathbb{C}P^n, J_0, g_0)$?
\end{question}
The answer to Question \ref{question} has been affirmatively verified in the following cases: ($p=0$, $n\leq 6$) and ($p=1$, $8\leq n\leq 51$) by Tanno in \cite{Ta1} and \cite{Ta2} respectively and ($p=2$, $n\neq 2, 8$) by Chen-Vanhecke in \cite{CV}. In \cite{Go} Goldberg attempted to treat the two exceptional cases $(p=2, n=2,8)$. In his later joint paper with Gauchman \cite{GG}, following the idea in \cite{Go}, they investigated Question \ref{question} for some other values of $(p,n)$ under some additional restrictions (\cite[Theorem 1]{GG}). In \cite{Pe} Perrone claimed to apply Kobayashi-Ochiai's characterization of $\mathbb{C}P^n$ mentioned above to give an alternative and unified proof for $(p=2, \text{all $n$})$. So now it seems to be widely believed that Question \ref{question} has been solved affirmatively in these cases and particularly for $p=2$, which is the only known case where the geometry of a compact K\"{a}hler manifold $(M,g)$ is completely determined by $\text{Spec}^p(M,g)$ for some fixed $p$ and in all dimensions $n$.

We need to point out that the proofs of the main results in \cite{Go}, \cite{GG} and \cite{Pe} contain \emph{gaps}. In \cite{Go} the treatment for $(p=2,n=2)$ is routine and correct by using the Gauss-Bonnet formula. However, the proof for $(p=2,n=8)$ is \emph{false} due to a mistaken volume estimate argument in \cite[p. 197-198]{Go}. In his later joint paper with Gauchman \cite{GG}, this false argument was carried over and formulated explicitly as a lemma in \cite[p. 566, Lemma 1]{GG}, on which the proof of the main result \cite[Theorem 1]{GG} relies crucially.
The mistake occurring in \cite{Pe} is due to a rescaling of the metric, which is forbidden after the spectrum set $\text{Spec}^2(M,g)$ is fixed.

 Our this paper has \emph{two main purposes}. \emph{The first one} is to point out precisely where the mistakes occur in \cite{Go} and \cite{Pe}. Although Goldberg's arguments in \cite[p. 197-198]{Go} for \cite[Lemma 1]{GG} are false, the arguments still contain very valuable information and can yield a conclusion weaker than what he claimed in \cite[Lemma 1]{GG}. Moreover this weaker conclusion should be a quite well-known fact in complex differential geometry but lacks a detailed proof, at least to the author's best knowledge. As Goldberg's arguments in \cite{Go} are quite sketchy, it deserves to present a detailed proof of this conclusion, which will be done in detail in Section \ref{section2} and should be of independent interest, and from this process we shall see where the mistake occurs in yielding a claimed proof of \cite[Lemma 1]{GG}. More interestingly, the content of \cite[Lemma 1]{GG} itself is \emph{correct} and is indeed equivalent to a long-standing conjecture in complex geometry solved very recently by Kento Fujita (\cite{Fu}), which shall also be explained in Section \ref{section2}.

Our \emph{second main purpose} in this paper is to solve Question \ref{question} for \emph{all positive even numbers $p$ in almost all dimensions $n$}. More generally, we shall prove the following result:

\begin{theorem}\label{main theorem}
Let $(M,J,g)$ be a compact K\"{a}hler manifold with $\text{Spec}^p(M,g)=\text{Spec}^p(\mathbb{C}P^n, g_0)$, where the pair of integers $(p,n)$ satisfies that $p$ be even, $2\leq p\leq 2(n-1)$ and
\be\label{relation}p^2-2np+\frac{n(2n-1)}{3}\neq0.\ee
Then $(M,J,g)$ is holomorphically isometric to $(\mathbb{C}P^n, J_0, g_0)$.
\end{theorem}

Relation (\ref{relation}) is equivalent to the non-vanishing of a coefficient arising from the Minakshisundaram asymptotic formula, which will be clear in Section \ref{section3}. Some arguments (Proposition \ref{lemmarelation}) show that the positive integer solutions $(n,p)$ to the equation
$$p^2-2np+\frac{n(2n-1)}{3}=0$$
with $p$ \emph{even} are precisely of the following forms
\begin{eqnarray}\label{equationsection1}
\left\{ \begin{array}{ll}
\text{$(n,p)=(n_k,p_k)$ or $(n_k,2n_k-p_k)$}\qquad(k=1,2,\cdots)\\
~\\
(n_1,p_1)=(48,20)\\
~\\
n_{k+1}=265n_k-168p_k+48\\
~\\
p_{k+1}=112n_k-71p_k+20.\\
\end{array} \right.
\end{eqnarray}
Easy calculations show that $(n_2,p_2)=(9408,3976)$, $(n_3,p_3)=(1825200,771420)$, $\cdots$, whose distributions are more and more sparse as $k\rightarrow\infty$.

Theorem \ref{main theorem}, together with (\ref{equationsection1}), yields the following result, which solves Question \ref{question} for all positive even numbers $p$ in almost all dimensions $n$:

\begin{theorem}\label{maincorollary}~
\begin{enumerate}
\item
For each positive even number $p\not\in\{p_k, 2n_k-p_k~|~k=1,2,\ldots\},$
Question \ref{question} holds in all dimensions $n$ with at most one possible exception $n=p/2$.

\item
If $p=p_k$ or $2n_k-p_k$ $(k=1,2,\ldots)$, Question \ref{question} holds in all dimensions $n$ with at most two possible exceptions $n=p/2$ or $n_k$.
\end{enumerate}
Here
\begin{eqnarray}
\left\{ \begin{array}{ll}
\{p_k, 2n_k-p_k~|~k=1,2,\ldots\}=\{20,76,3976, 14840,771420,2878980,\cdots\}\\
~\\
\{n_k~|~k=1,2,\ldots\}=\{48,9408,1825200,\ldots\},
\end{array} \right.\nonumber
\end{eqnarray}
whose distributions are more and more sparse as $k\rightarrow\infty$, are determined by (\ref{equationsection1}).
\end{theorem}

As we have mentioned above, Question \ref{question} has been rigorously proved for $(p=2, n\neq 8)$ in previous literature. However, the case $(p=2,n=8)$ is included in our Theorem \ref{maincorollary}. Thus we have \emph{rigorously} established the following result, which has been widely believed to be true for several decades.
\begin{corollary}
For $p=2$, Question \ref{question} holds in all dimensions $n$, which is the only known case where the geometry of a compact K\"{a}hler manifold $(M,g)$ is completely determined by $\text{Spec}^p(M,g)$ for \emph{some fixed $p$ and in all dimensions $n$}.
\end{corollary}

\subsection*{Outline of this paper}
The rest of this paper is organized as follows. We discuss in Section \ref{section2} the volume estimate problem for Fano K\"{a}hler-Einstein manifolds and present a detailed proof of a well-known result (Proposition \ref{volumeestimateweak}) following Goldberg's sketchy arguments in \cite[p. 197-198]{Go}, which is weaker than his claimed lemma in \cite[Lemma 1]{GG}. Through this process we shall see where his mistake occurs.
We will point out at the end of Section \ref{section2} the gap in \cite{Pe} as well.  Section \ref{section3} is devoted to some preliminaries on the proof of Theorem \ref{main theorem}: precise values of coefficients in front of pointwise squared norms of various tensors appearing in Riemannian and K\"{a}hler manifolds and their relations, and some integral formulas related to the first and second Chern classes for compact K\"{a}hler manifolds, on which the proof of Theorem \ref{main theorem} relies crucially. After these preliminaries, we shall prove Theorem \ref{main theorem} and Corollary \ref{maincorollary} in Section \ref{section4}. During the process of proving Theorems \ref{main theorem} and \ref{maincorollary}, in addition to the preliminaries in Section \ref{section3}, we either need two very technical results: Propositions \ref{technicallemma1} and \ref{lemmarelation}. To avoid digressing from illustrating the main ideas of the proofs of our main results in Section \ref{section4}, we postpone the proofs of these two technical results to the last Section \ref{section5} entitled ``Appendix".

\section*{Acknowledgements}
I would like to thank Yinhe Peng and Wei Xu for finding out the paper \cite{Go} from Canada and sending it to me.

\section{Volume estimate for Fano K\"{a}hler-Einstein manifolds}\label{section2}
Before starting the main contents of this section, let us make some conventions, which will be frequently used in the sequel.

Suppose $(M,J,g)$ is a complex $n$-dimensional compact K\"{a}hler manifold. Then we define
\begin{eqnarray}\label{kahlerricciformtensor}
\left\{ \begin{array}{ll}
\text{$\omega:=\frac{1}{2\pi}g(J\cdot,\cdot)$, the K\"{a}hler form of $g$,}\\
~\\
\text{Ric$(g)$:=the Ricci tensor of $g$,}\\
~\\
\text{Ric$(\omega):=\frac{1}{2\pi}
$Ric$(g)(J\cdot,\cdot)$, the Ricci form of $g$,}\\
~\\
\text{$s_g:={\rm Trace}_g{\rm Ric}(g)$, the scalar curvature of $g$.}
\end{array} \right.
\end{eqnarray}

It is well-known that $\text{Ric$(\omega)$}$ represents the first Chern class of $(M,J)$ and
\be\label{volumeelement}\text{the volume element of $(M,g)=\frac{\pi^n}{n!}\omega^n$}\ee
in our notation of $\omega$.

\subsection{Backgrounds and results on volume estimate}
We assume throughout this subsection that $X$ is an $n$-dimensional Fano manifold, i.e., a compact complex manifold of complex dimension $n$ whose first Chern class $c_1(X)>0$. This means that $c_1(X)$ can be represented by a K\"{a}hler metric/form and particularly $X$ is a projective algebraic manifold.
The \emph{Fano index} of $X$ is
defined to be the largest positive integer $I$ such that $c_1(X)$ is divisible by $I$, i.e.,
$c_1(X)/I\in H^2(X;\mathbb{Z}).$ We denote by $I(X)$ the Fano index
of $X$. The importance of Fano index is due to a classical result of Kobayashi and Ochiai, who showed in \cite{KO} that  $I(X)\leq n+1$, with equality if and only if $X\cong\mathbb{C}P^n$. Here ``$\cong$" denotes ``biholomorphic to".

 Note that in this case the special Chern number $c_1^n[X]$, also called the top intersection number or the degree of $X$, is a positive integer. For some time it was conjectured that $c_1^n[X]$ is bounded above by $(n+1)^n$ and can only be attained by $\mathbb{C}P^n$ (\cite[p. 133]{yau2}). Namely,
\be\label{conjecture}c_1^n[X]\leq(n+1)^n,\qquad \text{with equality if and only if
$X\cong\mathbb{C}P^n$}.\ee
Now there have been many counterexamples to this bound. See for example \cite[p. 128]{IP}. Debarre
constructed in \cite[p. 137-139]{De} a family of Fano manifolds and
used them to illustrate that there is indeed \emph{no} universal
polynomial upper bound on $\sqrt[n]{c_1^n[X]}$ among all the
$n$-dimensional Fano manifolds $X$.

S.-T. Yau's celebrated Chern number inequalities for compact K\"{a}hler manifolds with $c_1<0$ or $c_1=0$ (\cite{Yau}) are based on the existence of K\"{a}hler-Einstein metrics in these two cases. With this fact in mind, although (\ref{conjecture}) is not true for general Fano manifolds $X$, it is still natural to ask if (\ref{conjecture}) holds for any Fano manifold $X$ admitting a K\"{a}hler-Einstein metric. Berman-Berndtsson showed in \cite{BB} that this is true when such an $X$ either admits a holomorphic $\mathbb{C}^{\ast}$-action with isolated fixed points or is toric. Very recently Kento Fujita proved this result in its full generality (\cite{Fu}):
\begin{theorem}[Kento Fujita]\label{theoremfujita}
Let $X$ be an $n$-dimensional Fano manifold admitting a K\"{a}hler-Einstein metric. Then (\ref{conjecture}) holds.\end{theorem}
If $X$ is a Fano manifold admitting a K\"{a}hler-Einstein metric $g$, then by definition we have
$$\text{Ric}(\omega)=\frac{s_g}{2n}\omega$$
with positive constant scalar curvature $s_g$,
 which via (\ref{volumeelement}) implies that
 $$c_1^n[X]=\frac{n!\cdot s_g^n}{(2n\pi)^n}\text{Vol}(M,g).$$  Also note that $(\mathbb{C}P^n, J_0, g_0)$ is a Fano K\"{a}hler-Einstein manifold and a classical result of Berger (cf. \cite[p. 74]{LB}) tells us that any two K\"{a}hler-Einstein metrics on $(\mathbb{C}P^n, J_0)$ are proportional. Thus Theorem \ref{theoremfujita} has the following equivalent form, which is exactly what Goldberg claimed in \cite[Lemma 1]{GG}.
\begin{theorem}[$\Leftrightarrow$Theorem \ref{theoremfujita}]\label{theorem2fujita}
Let $(X,g)$ be an $n$-dimensional Fano K\"{a}hler-Einstein manifold such that $s_g=s_{g_0}$. Then $\text{Vol}(X,g)\leq\text{Vol}(\mathbb{C}P^n,g_0)$,
where the equality holds if and only if $(X,g)$ is holomorphically isometric to $(\mathbb{C}P^n, J_0, g_0)$.
\end{theorem}

The proof of Theorem \ref{theoremfujita} in \cite{Fu} is algebraic and relies heavily on recently developed deep results related to the notion of K-stability. So it is natural to ask if one can give a purely complex differential geometric proof of Theorem \ref{theoremfujita}. Although this aim seems to be out of reach with currently available tools, we can still apply purely complex differential geometric results to yield the following weaker estimate:
\begin{proposition}\label{volumeestimateweak}
Let $X$ be an $n$-dimensional Fano manifold admitting a K\"{a}hler-Einstein metric. Then \be\label{weakerresult}c_1^n[X]\leq\frac{n+1}{I(X)}\cdot(n+1)^n, \nonumber\ee
with equality if and only if
$X\cong\mathbb{C}P^n$. Or equivalently, let $(X,g)$ be an $n$-dimensional Fano K\"{a}hler-Einstein manifold such that the scalar curvature of $g$ is equal to that of $g_0$. Then $$\text{Vol}(X,g)\leq\frac{n+1}{I(X)}\cdot\text{Vol}(\mathbb{C}P^n,g_0),$$
and the equality holds if and only if $(X,g)$ is holomorphically isometric to $(\mathbb{C}P^n, J_0, g_0)$.
\end{proposition}

\begin{remark}~
\begin{enumerate}
\item
Clearly Proposition \ref{volumeestimateweak} is weaker than Theorems \ref{theoremfujita} and \ref{theorem2fujita} as the Fano index $I(X)\leq n+1$.

\item
In \cite[\S 2.3]{GMSY} the authors sketchily outline a proof of Proposition \ref{volumeestimateweak} and in \cite[p. 1]{BB} Berman-Berndtsson credits it to \cite{GMSY}. Indeed Proposition \ref{volumeestimateweak} should be well-known to experts for at least several decades. For example, in \cite[p. 125]{LS} Proposition \ref{volumeestimateweak} was stated as a well-known fact without a proof.
\end{enumerate}
\end{remark}

\subsection{Proof of Proposition \ref{volumeestimateweak}}
We present in this subsection a detailed proof of Proposition \ref{volumeestimateweak} following Goldberg's sketchy arguments in \cite{Go}. As we have mentioned above, Proposition \ref{volumeestimateweak} has been well-known to experts for at least several decades. Nevertheless, it lacks a detailed proof in the existing literature, at least to the author's best knowledge. So this subsection may be of independent interest to differential geometry experts and particularly those interested in the proof of Proposition \ref{volumeestimateweak} via a purely geometric method rather than as a corollary of Theorem \ref{theoremfujita}.

A key ingredient in proving Proposition \ref{volumeestimateweak} is an improvement of a result of Kobayashi (\cite[p. 136, Theorem 5]{Ko}). So let us start by recalling some related materials in \cite{Ko}. The primary purpose of \cite{Ko} is to show that if the curvature of a complete K\"{a}hler manifold does not deviate much from that of a complex projective space, then the homotopy groups of this manifold are the same as those of this complex projective space. To achieve this aim, he needs to construct a principal circle bundle over the manifold in question and then apply the homotopy exact sequence. So in \cite[\S 3]{Ko} the Riemannian structure on a circle bundle was carefully investigated and as a byproduct he obtained \cite[p. 136, Theorem 5]{Ko}, which is what we need to prove Proposition \ref{volumeestimateweak}. The following materials are basically taken from \cite[\S3]{Ko}.

Suppose $(M,g)$ is an $m$-dimensional Riemannian manifold with Riemannian metric $g$. Locally, $$g=\sum_{i=1}^m(\theta^i)^2,$$
 where $\{\theta^1,\ldots,\theta^m\}$ is a coframe field defined on some open subset in $M$. Let $K_{ijkl}$ $(1\leq i,j,k,l\leq m)$ be the components of the curvature tensor with respect to $\{\theta^i\}$. Note that in different literature the notation $K_{ijkl}$ maybe ambiguous up to a sign. Here we follow the notation in \cite[\S 3]{Ko} to define $K_{jikl}$ in such a manner that the components of the Ricci tensor, denoted by $K_{ij}$, and the scalar curvature $s_g$ are given by
\begin{eqnarray}\label{riccitensornotation}
\left\{ \begin{array}{ll} K_{ij}:=\sum_kK_{ikjk},\\
~\\
s_g:=\sum_iK_{ii}=\sum_{i,k}K_{ikik}.
\end{array}\right.
\end{eqnarray}

Let $S^1\hookrightarrow P\xlongrightarrow{\pi}M$ be a principal circle bundle over $M$ and $\gamma$ its connection form, which is a one-form on the total space $P$. Then $\Gamma:=\text{d}\gamma$ is the curvature form of this principal circle bundle, which is a two-form on $P$.
By transgression, $\Gamma$ can be written as \be\label{gamma}\Gamma=\pi^{\ast}(\sum_{i,j} A_{ij}\theta^i\wedge\theta^j),\qquad
A_{ij}=-A_{ji}.\nonumber\ee

We now construct a family of Riemannian metrics $h(a)$ on $P$ from $g$, which are parameterized by a positive number $a$, as follows:
\be\label{metricpara}h(a):=\pi^{\ast}(g)+(a\gamma)^2,\qquad a>0,\ee
i.e., if we set $$\varphi^0:=a\gamma, \qquad\varphi^i:=\pi^{\ast}(\theta^i),\qquad(1\leq i\leq m),$$ then $$h(a)=(\varphi^0)^2+\sum_{i=1}^m(\varphi^i)^2,$$
i.e., $\{\varphi^0,\varphi^1,\ldots,\varphi^m\}$ is a coframe field of $h(a)$ and $\varphi^0$ corresponds to the direction along the fiber. We denote by $R_{ijkl}$ $(0\leq i,j,k,l\leq m)$ the components of the curvature tensor of $h(a)$ with respect to the coframe field $\{\varphi^0,\varphi^1,\ldots,\varphi^m\}$. Note that the volumes of $(M,g)$ and $(P, h(a))$ are related by (cf. \cite[p. 317]{BG})
\be\label{volumerelation}\text{Vol}(P,h(a))=\text{Vol}(M,g)\cdot2\pi a.\ee

The components of the two curvature tensors $K_{ijkl}$ and $R_{ijkl}$ can be related by $A_{ij}$ and the parameter $a$ as follows (\cite[p. 126]{Ko}):

\begin{lemma}\label{lemmacurvaturetensor}
\begin{eqnarray}\label{curvaturetensor}
\left\{ \begin{array}{ll} R_{ijkl}=K_{ijkl}-a^2(2A_{ij}A_{kl}+A_{ik}A_{jl}-A_{il}A_{jk}),&
(1\leq i,j,k,l\leq m)\\
~\\
R_{i0k0}=a^2\sum_{l}A_{il}A_{kl}, &(1\leq i,k\leq m)\\
~\\
R_{i0kl}=-aA_{kl;i},&(1\leq i,k,l\leq m).
\end{array}\right.
\end{eqnarray}
Here $A_{ij;k}$ are precisely the covariant derivatives of the tensor field $A_{ij}$ with respect to the Riemannian connection of $(M,g)$. For our later purpose we only need the fact that $A_{ij;k}\equiv0$ if all these $A_{ij}$ are constants.
\end{lemma}
As in (\ref{riccitensornotation}), we denote by $R_{ij}$ and $K_{ij}$ the components of the Ricci tensors $\text{Ric}(g)$ and $\text{Ric}(h(a))$ respectively, i.e.,
\begin{eqnarray}
\left\{ \begin{array}{ll} \text{Ric}(g)=\sum_{1\leq i,j\leq m}K_{ij}\theta^i\theta^j,\\
~\\
\text{Ric}(h(a))=\sum_{0\leq i,j\leq m}R_{ij}\varphi^i\varphi^j.
\end{array}\right.
\nonumber
\end{eqnarray}

Then they are related as follows:

\begin{lemma}\label{lemmariccitensor}
\begin{eqnarray}\label{riccitensor}
\left\{ \begin{array}{ll} R_{ij}=K_{ij}+2a^2\sum_{1\leq k\leq m}A_{ik}A_{kj},&
(1\leq i,j\leq m)\\
~\\
R_{00}=-a^2\sum_{1\leq i,j\leq m}A_{ij}A_{ji},\\
~\\
R_{i0}=a\sum_{1\leq k\leq m}A_{ik;k},&(1\leq i\leq m).
\end{array}\right.
\end{eqnarray}
\end{lemma}
\begin{proof}
Direct calculations via (\ref{curvaturetensor}) and the convention assumed in (\ref{riccitensornotation}) yield (\ref{riccitensor}).
\end{proof}

Now we assume that $M$ is a complex $n$-dimensional manifold of complex structure $J$ and $g$ a K\"{a}hler metric. Then $m=2n$ and the coframe field $\{\theta^1,\ldots,\theta^{2n}\}$ can be written as the form $\{\theta^i,J\theta^i,~1\leq i\leq n\}$, i.e., $\theta^{i+n}=J\theta^i$. In this case
$$\Big\{X^i:=\frac{1}{\sqrt{2}}(\theta^i+\sqrt{-1}J\theta^i)~\big|~ 1\leq i\leq n\Big\}$$
is a $(1,0)$-type unitary coframe field and thus the K\"{a}hler form $\omega$ \big(recall (\ref{kahlerricciformtensor})\big) is given by
\be\label{kahlerformexpression}
\begin{split}
\omega&=\frac{\sqrt{-1}}{2\pi}\sum_{i=1}^nX^i\wedge\bar{X^i}\\
&=\frac{\sqrt{-1}}{2\pi}\sum_{i=1}^n\big(\frac{1}{\sqrt{2}}(\theta^i+\sqrt{-1}J\theta^i)\big)\wedge
\big(\frac{1}{\sqrt{2}}(\theta^i-\sqrt{-1}J\theta^i)\big)\\
&=\frac{1}{4\pi}\sum_{i=1}^n(\theta^i\wedge J\theta^i-J\theta^i\wedge\theta^i)\\
&=\frac{1}{4\pi}\sum_{i=1}^n(\theta^i\wedge \theta^{i+n}-\theta^{i+n}\wedge\theta^i).
\end{split}\ee

If moreover $(M,J,g)$ is K\"{a}hler-Einstein, then $$\text{Ric}(\omega)=\frac{s_g}{2n}\omega$$ represents $c_1(M)$. It is well-known that $\omega$ is a harmonic form with respect to the metric $g$. Assume that $I$ is a positive integer such that $c_1(M)/I\in H^2(M;\mathbb{Z})$. This means
\be\label{harmonic1}\frac{s_g}{2nI}\omega\ee
is a harmonic form representing \be\label{harmonic2}\frac{c_1(M)}{I}\in H^2(M;\mathbb{Z}).\ee

By \cite[p. 131, Prop.9]{Ko}, which has now become a standard fact, there exists a principle circle bundle
$S^1\hookrightarrow P\xlongrightarrow{\pi}M$ and a connection form $\gamma$ on $P$ such that
$$\text{d}\gamma=\pi^{\ast}(\frac{s_g}{2n I}\omega)\stackrel{(\ref{kahlerformexpression})}{=}\pi^{\ast}
\big(\frac{s_g}{8n\pi I}\sum_{i=1}^n(\theta^i\wedge \theta^{i+n}-\theta^{i+n}\wedge\theta^i)\big).$$

This means in the case of $g$ being K\"{a}hler-Einstein, we may take
$$\sum_{i,j=1}^{2n}A_{ij}\theta^i\wedge\theta^j=\frac{s_g}{8n\pi I}\sum_{i=1}^n(\theta^i\wedge \theta^{i+n}-\theta^{i+n}\wedge\theta^i),$$
i.e.,
$$(A_{ij})=\frac{s_g}{8n\pi I}\begin{pmatrix}0&I_n\\
-I_n&0\end{pmatrix},\qquad\text{$I_n=$ rank $n$ identity matrix,}$$
which implies that
$$(A_{ij})^2=-\big(\frac{s_g}{8n\pi I}\big)^2I_{2n}.$$

Summarizing the above discussions,
we have proved the following result:

\begin{lemma}\label{aij}
 Suppose $(M,J,g)$ is a K\"{a}hler-Einstein manifold of complex dimension $n$ and $I$ a positive integer such that $c_1(M)/I\in H^2(M;\mathbb{Z})$. There exists a principle circle bundle
$S^1\hookrightarrow P\xlongrightarrow{\pi}M$ and a connection form $\gamma$ on $P$ such that the corresponding $A_{ij}$ satisfy
\be\label{kroneckerdelta}\sum_{k=1}^{2n}A_{ik}A_{kj}=
-\big(\frac{s_g}{8n\pi I}\big)^2\delta_{ij},\ee
where $\delta_{ij}$ is the Kronecker delta. Moreover, in this case all $A_{ij;k}\equiv0$ as $A_{ij}$ are constants.
\end{lemma}

With Lemmas \ref{lemmariccitensor} and \ref{aij} in hand, we can show the following result, which is a quantitative version of \cite[p. 136, Theorem 5]{Ko}.

\begin{theorem}\label{kobayashitheorem}
Suppose $(M,J,g)$ is a complex $n$-dimensional Fano K\"{a}hler-Einstein manifold, $I(M)$ the Fano index of $M$, and $S^1\hookrightarrow P\xlongrightarrow{\pi}M$ the principal circle bundle over $M$ corresponding to $c_1(M)/I(M)\in H^2(M;\mathbb{Z})$. Then the components $R_{ij}$ of the Ricci tensor of the $(2n+1)$-dimensional Riemannian manifold $(P,h(a))$ constructed in (\ref{metricpara}) satisfy
\begin{eqnarray}\label{riccitensoreinstein}
\left\{ \begin{array}{ll} R_{ij}=\big[\frac{s_g}{2n}-\frac{2s_g^2a^2}{\big(8n\pi I(M)\big)^2}\big]\delta_{ij},&
(1\leq i,j\leq 2n)\\
~\\
R_{00}=\frac{s_g^2a^2}{2n\big(4\pi I(M)\big)^2},\\
~\\
R_{i0}=0,&(1\leq i\leq 2n).
\end{array}\right.
\end{eqnarray}

Consequently, the metric $$h(\frac{\sqrt{n}4\pi I(M)}{\sqrt{s_g(n+1)}})$$ is Einstein with scalar curvature $\frac{s_g(2n+1)}{2(n+1)}$. Moreover,
$$\big(P, h(\frac{\sqrt{n}4\pi I(M)}{\sqrt{s_g(n+1)}})\big)$$
is isometric to standard $(2n+1)$-dimensional sphere with constant sectional curvature $\frac{s_g}{4n(n+1)}$, denoted by $\mathbb{S}^{2n+1}(\frac{s_g}{4n(n+1)})$, if and only if $(M,J,g)$ is holomorphically isometric to $(\mathbb{C}P^n,J_0,g_0)$ with $s_{g_0}=s_g$.
\end{theorem}

\begin{proof}
Substituting (\ref{kroneckerdelta}) into (\ref{riccitensor}) and noticing the facts that $K_{ij}=\frac{s_g}{2n}\delta_{ij}$ and $A_{ij;k}=0$ in this case yield (\ref{riccitensoreinstein}). As $(M,J)$ is Fano, this means the (constant) scalar curvature $s_g$ is \emph{positive} and thus the following equation arising from (\ref{riccitensoreinstein})
$$\frac{s_g}{2n}-\frac{2s_g^2a^2}{\big(8n\pi I(M)\big)^2}=\frac{s_g^2a^2}{2n\big(4\pi I(M)\big)^2}$$
is solvable for variable $a>0$:
$$a=\frac{\sqrt{n}4\pi I(M)}{\sqrt{s_g(n+1)}},$$
in which case we have
$$R_{ij}=\frac{s_g}{2(n+1)}\delta_{ij},\qquad 0\leq i,j\leq 2n.$$

This means the metric$$h(\frac{\sqrt{n}4\pi I(M)}{\sqrt{s_g(n+1)}})$$ is Einstein whose scalar curvature is $$(2n+1)\cdot\frac{s_g}{2(n+1)}=\frac{s_g(2n+1)}{2(n+1)}.$$

The second part of this theorem has been established in \cite[p. 136]{Ko} after the statement of Theorem 5.
\end{proof}
\begin{remark}
Note that if an $m$-dimensional manifold has constant sectional curvature $c$, then its scalar curvature is precisely $cm(m-1)$. Thus the scalar curvature of $\mathbb{S}^{2n+1}(\frac{s_g}{4n(n+1)})$ is $\frac{s_g(2n+1)}{2(n+1)}$.
\end{remark}

With this established Theorem \ref{kobayashitheorem} in hand, we can now proceed to prove Proposition \ref{volumeestimateweak}, which is a beautiful application of Bishop's volume estimate formula.

\begin{proof}
Suppose $(M,J,g)$ is a complex $n$-dimensional Fano K\"{a}hler-Einstein manifold. Then Theorem \ref{kobayashitheorem} and (\ref{volumerelation}) tell us that
\be\label{1}\text{Vol}\big(P, h(\frac{\sqrt{n}4\pi I(M)}{\sqrt{s_g(n+1)}})\big)=\text{Vol}(M,g)\cdot2\pi\frac{\sqrt{n}4\pi I(M)}{\sqrt{s_g(n+1)}}.\ee

By Bishop's volume estimate (cf. \cite[p. 256, Coro.4]{BC} or \cite[p. 11]{SY}) we have
\be\label{2}\text{Vol}\big(P, h(\frac{\sqrt{n}4\pi I(M)}{\sqrt{s_g(n+1)}})\big)\leq \text{Vol}\big(\mathbb{S}^{2n+1}(\frac{s_g}{4n(n+1)})\big),\ee
with equality if and only if $$\big(P, h(\frac{\sqrt{n}4\pi I(M)}{\sqrt{s_g(n+1)}})\big)$$ is isometric to $\mathbb{S}^{2n+1}(\frac{s_g}{4n(n+1)})$.
Now we apply Theorem \ref{kobayashitheorem} to $(M,J,g)=(\mathbb{C}P^n,J_0,g_0)$ with $s_{g_0}=s_g$. Again by (\ref{volumerelation})
we have
\be\label{3}\text{Vol}\big
(\mathbb{S}^{2n+1}(\frac{s_{g_0}}{4n(n+1)})\big)
=\text{Vol}\big(\mathbb{C}P^n, g_0)\cdot2\pi\frac{\sqrt{n}4\pi I(\mathbb{C}P^n)}{\sqrt{s_{g_0}(n+1)}}.\ee

Combining (\ref{1}), (\ref{2}), (\ref{3}), the second part of Theorem \ref{kobayashitheorem} and the facts that $s_g=s_{g)}$ and $I(\mathbb{C}P^n)=n+1$, we have
\be\label{4}\text{Vol}(M,g)\leq\frac{n+1}{I(M)}\cdot\text{Vol}(\mathbb{C}P^n,g_0),\ee
with equality if and only if $(M,J,g)$ is holomorphically isometric to $(\mathbb{C}P^n, J_0, g_0)$.
\end{proof}
\begin{remark}
We can easily see from (\ref{harmonic1}) and (\ref{harmonic2}) that the role of the Fano index $I(M)$ can be replaced by any positive integer $I$ such that $c_1(M)/I\in H^2(M;\mathbb{Z})$ and then the last upper bound we can obtain accordingly become
$$\frac{n+1}{I}\cdot\text{Vol}(\mathbb{C}P^n,g_0).$$
So the choice of $I(M)$ is to make this upper bound as sharp as possible.
\end{remark}

\subsection{Gaps in some previous literature}
We now point out the mistakes in \cite[p. 197-198]{Go} as well as in \cite{Pe}.

The first (minor) mistake is the inaccuracy of some constants. For instance, the constant sectional curvature of the sphere is $\frac{s_g}{4n(n+1)}$ rather than $\frac{(2n+1)s_g}{4n(n+1)}$ ($s_g$ was denoted by $\rho$ in \cite{Go}). But this mistake is not essential. Note that the volume relation in (\ref{volumerelation}) related to $M$ and $P$
depends on the choice of the parameter $a$, and the choices of $a$ in (\ref{1}) and (\ref{3}) for abstract $M$ and concrete $\mathbb{C}P^n$ are different. This fact was completely ignored in \cite{Go} and thus lead to his claimed \cite[Lemma 1]{GG}, which is the essential mistake made in \cite{Go}.

Perrone attempted in \cite{Pe} to apply Kobayashi-Ochiai's characterization in \cite{KO} to give an alternative proof of Question \ref{question} for $p=2$. That is, he tried to show that the first Chern class of the manifold $M$ under consideration satisfies $c_1(M)=(n+1)t$, where $t$ is the positive generator of $H^2(M,\mathbb{Z})$. However, he wrote at the end of \cite[p. 872]{Pe} that ``\emph{Therefore, multiplying the K\"{a}hler metric $g$ by some constant, if necessary, we may assume that $\omega=[\phi]$ is a positive element of $H^{1,1}(M,\mathbb{Z})$}". This claim is \emph{false} as the rescaling of the metric $g$ here is not allowed, which can be easily seen from the following lemma.

\begin{lemma}\label{laplacianscaling}
Suppose $(M,g)$ is a compact Riemannian manifold and $\Delta_g$ denotes the Laplacian of $M$ w.r.t. $g$. If $\lambda$ is a positive constant, then $\Delta_{\lambda g}=\lambda\Delta_g$, and consequently for any $p$ we have \be\label{laplacianscaling2}{\rm Spec}^p(M,\lambda g)=\lambda{\rm Spec}^p(M,g).\ee
\end{lemma}
\begin{proof}
We denote for convenience by $<,>_g$ the inner product on the exterior differential forms induced from $g$. Note that if $\alpha_1,\alpha_2\in\Omega^p(M)$, then $$<\alpha_1,\alpha_2>_{\lambda g}=\lambda^p<\alpha_1,\alpha_2>_g.$$

Denote by $d^{\ast}_g$ and $d^{\ast}_{\lambda g}$ the formal adjoint of the $d$-operator w.r.t. the metrics $g$ and $\lambda g$ respectively. Suppose $\alpha\in\Omega^p(M)$ and $\beta\in\Omega^{p+1}(M)$. Then
\be\label{scale1}<\text{d}\alpha,\beta>_{\lambda g}=<\alpha,d^{\ast}_{\lambda g}\beta>_{\lambda g}=
\lambda^{p}<\alpha,d^{\ast}_{\lambda g}\beta>_{g}.\ee

On the other hand,
\be\label{scale2}<\text{d}\alpha,\beta>_{\lambda g}=\lambda^{p+1}<\text{d}\alpha,\beta>_{g}=
\lambda^{p+1}<\alpha,d^{\ast}_{g}\beta>_{g}=
\lambda^{p}<\alpha,(\lambda d^{\ast}_{g})\beta>_{g}.\ee

Combining (\ref{scale1}) and (\ref{scale2}) yields $d^{\ast}_{\lambda g}=\lambda d^{\ast}_{g}$ and therefore $$\Delta_{\lambda g}=d\circ d^{\ast}_{\lambda g}+d^{\ast}_{\lambda g}\circ d=d\circ \lambda d^{\ast}_{g}+\lambda d^{\ast}_{g}\circ d=\lambda\Delta_g.$$
\end{proof}

(\ref{laplacianscaling2}) tells us that when rescaling the metric $g$, the spectral set $\text{Spec}^p(M,g)$ rescales accordingly. Therefore rescaling the metric $g$ in \cite{Pe} is NOT allowed as the initial requirement is that $\text{Spec}^2(M,g)=\text{Spec}^2(\mathbb{C}P^n,g_0)$ and thus fixed. The same mistake also occured in \cite[p. 346]{Pe2}, where he attempted to apply the same idea to treat the hyperquadrics as there is a similar characterization for hyperquadrics given by Kobayashi-Ochiai in \cite{KO}.

\section{Preliminaries}\label{section3}
This section contains necessary preliminaries on the proof of Theorem \ref{main theorem} in Section \ref{section4}. To be more precise, we are concerned with in subsection \ref{subsection3.1} the pointwise squared norms of curvature tensors and  their orthogonal components under $L^2$-norms for compact Riemannian and K\"{a}hler manifolds respectively, and their relations. The relations between the first and second Chern classes and the pointwise squared norms of various tensors for compact K\"{a}hler manifolds shall be discussed as well in subsection \ref{subsection3.2}. The materials in subsection \ref{subsection3.1} should be well-known to experts, but we are not able to find a reference where these results are stated as clear and precise as ours, at least to the author's best knowledge. So the author hopes that these materials can also be used as an individual reference to interested readers in this field.

\subsection{Norms of various tensors arising from curvature}\label{subsection3.1}
It is well-known that (\cite[p. 45]{Be}) the curvature tensor $R$ of an $m$-dimensional Riemannian manifold $(M,g)$ splits naturally into
three irreducible components under the orthogonal group: \be\label{decompostion1}R=S+P+W,\ee
 where $S$, $P$ and $W$ involve the scalar curvature part, the traceless Ricci tensor part and the Weyl curvature part respectively and $W$ exists as a nontrivial summand only when $m\geq 4$. The metric $g$ is of constant sectional curvature or Einstein if and only if  $P=W=0$ or $P=0$ respectively. Moreover, the decomposition $R=S+P+W$ is orthogonal under the norm of these tensors defined below.

 Now under some local coordinates $(x^1,\ldots,x^m)$,
we denote the components of the metric tensor $g$, the Ricci tensor $\text{Ric}(g)$ and the traceless Ricci tensor $$\tilde{\text{R}}\text{ic}(g):=\text{Ric}(g)-\frac{s_g}{m}g$$ by $g_{ij}$, $R_{ij}$ and $$\tilde{R}_{ij}:=R_{ij}-\frac{s_g}{m}g_{ij}$$ respectively. The metric $g$ is Einstein if and only if $\tilde{\text{R}}\text{ic}(g)\equiv0$. When viewing $R$ as a $(0,4)$-type tensor, (\ref{decompostion1}) reads under local coordinates as follows:
 \be\label{localdecomp1}R_{ijkl}=S_{ijkl}+P_{ijkl}+W_{ijkl},\ee
 where
 \begin{eqnarray}\label{localdecomp11}
\left\{ \begin{array}{ll} S_{ijkl}=\frac{s_g}{m(m-1)}(g_{il}g_{jk}-g_{ik}g_{jl}),\\
~\\
P_{ijkl}=\frac{1}{m-2}(g_{il}\tilde{R}_{jk}-g_{ik}
\tilde{R}_{jl}+g_{jk}\tilde{R}_{il}-g_{jl}\tilde{R}_{ik}),\\
~\\
W_{ijkl}:=R_{ijkl}-S_{ijkl}-P_{ijkl}.
\end{array}\right.
\end{eqnarray}

Denote by $g^{ij}$ the entries of the inverse matrix of $(g_{ij})$: $(g^{ij}):=(g_{ij})^{-1}$. Then the pointwise squared norms of these tensors are defined as follows:
\begin{eqnarray}\label{localdecomp11}
\left\{ \begin{array}{ll} |R|^2:= R_{ijkl}R_{pqrs}g^{ip}g^{jq}g^{kr}g^{ls},\\
~\\
|\text{Ric}(g)|^2:=R_{ij}R_{pq}g^{ip}g^{jq},\\
~\\
\cdots.
\end{array}\right.\nonumber
\end{eqnarray}
Here and henceforth we sometimes adopt the Einstein convention for summation.

It is well-known that these norms are independent of the choice of the local coordinates and thus globally defined. These pointwise squared norms satisfy the following well-known facts:
\begin{lemma}\label{normrelation1}
If $(M,g)$ is an $m$-dimensional Riemannian manifold, then
\begin{eqnarray}\label{normrelation11}
\left\{ \begin{array}{ll} |S|^2=\frac{2s_g^2}{m(m-1)},&(m\geq 2)\\
~\\
|P|^2=\frac{4}{m-2}|{\rm\tilde{R}ic}(g)|^2,&(m\geq 3)\\
~\\
|{\rm Ric}(g)|^2=|{\rm\tilde{R}ic}(g)|^2+\frac{s_g^2}{m}.\\
\end{array}\right.
\end{eqnarray}
In particular, $|{\rm Ric}(g)|^2\geq\frac{s_g^2}{m}$ and with equality if and only if $g$ is Einstein.
\end{lemma}

\begin{lemma}
Although (\ref{normrelation11}) shall not be directly used in proving our main result, we still state it here for both completeness and the reader's convenience, and a comparison with Lemma \ref{normrelation2} and Proposition \ref{relationbetweenrieka} below, which will play key roles in proving Theorem \ref{main theorem}.
\end{lemma}

Now we further assume that this $(M,g)$ is a complex $n$-dimensional manifold (thus $m=2n$) and $g$ a K\"{a}hler metric. Then the K\"{a}hler curvature tensor, which is the complexification of the Riemannian curvature tensor and denoted by $R^{c}$, also splits into three irreducible components under unitary group bearing some
resemblance to (\ref{decompostion1}) (\cite[p. 77]{Be}):
 \be\label{decompostion2}R^{c}=S^{c}+P^{c}+B,\ee
where $S^{c}$, $P^{c}$ and $B$ involve respectively the scalar curvature part, the traceless Ricci tensor part and what has now become known as the Bochner curvature tensor.
$g$ is of constant holomorphic sectional curvature or Einstein if and only if $P^c=B\equiv0$ or $P^c\equiv0$ respectively. Moreover, the decomposition (\ref{decompostion2}) is orthogonal with respect to the norms defined below.

Under the local complex coordinates $(z^1,\ldots,z^n)$,
we write the K\"{a}hler form $\omega$, the Ricci form $\text{Ric}(\omega)$, the traceless Ricci form $\tilde{\text{R}}\text{ic}(\omega)$ and the $(0,4)$-type K\"{a}hler curvature tensor $R^c$ as follows:
\begin{eqnarray}
\left\{ \begin{array}{ll}
\omega=\frac{\sqrt{-1}}{2\pi}g(J\cdot,\cdot)=
\frac{\sqrt{-1}}{2\pi}g_{i\bar{j}}\text{d}z^i\wedge\text{d}\bar{z}^j\\
~\\
\text{Ric}(\omega)=\frac{\sqrt{-1}}{2\pi}
\text{Ric}(g)(J\cdot,\cdot)=\frac{\sqrt{-1}}{2\pi}
R_{i\bar{j}}\text{d}z^i\wedge\text{d}\bar{z}^j\\
~\\
\tilde{\text{R}}\text{ic}(\omega):=\text{Ric}(\omega)-\frac{s_g}{2n}\omega
=\frac{\sqrt{-1}}{2\pi}
(R_{i\bar{j}}-\frac{s_g}{2n}g_{i\bar{j}})\text{d}z^i\wedge\text{d}\bar{z}^j
=:\frac{\sqrt{-1}}{2\pi}\tilde{R}_{i\bar{j}}\text{d}z^i\wedge\text{d}\bar{z}^j
\\
~\\
R_{i\bar{j}k\bar{l}}^c:=R(\frac{\partial}{\partial
z_i},\frac{\partial}{\partial \bar{z_j}},\frac{\partial}{\partial
z_k},\frac{\partial}{\partial \bar{z_l}}).
\end{array}\right.
\end{eqnarray}

Now (\ref{decompostion2}) reads under local complex coordinates as follows (cf. \cite[p. 86]{Bo}):
\be\label{localdecomp2}
R_{i\bar{j}k\bar{l}}^c=
S_{i\bar{j}k\bar{l}}^c+P_{i\bar{j}k\bar{l}}^c+B_{i\bar{j}k\bar{l}},\ee
where
\begin{eqnarray}\label{localdecomp22}
\left\{ \begin{array}{ll}
S_{i\bar{j}k\bar{l}}^c=\frac{s_g}{2n(n+1)}(g_{i\bar{j}}g_{k\bar{l}}+g_{i\bar{l}}g_{k\bar{j}}),\\
~\\
P_{i\bar{j}k\bar{l}}^c=\frac{1}{n+2}
(g_{i\bar{j}}\tilde{R}_{k\bar{l}}+
g_{k\bar{l}}\tilde{R}_{i\bar{j}}+g_{i\bar{l}}
\tilde{R}_{k\bar{j}}+g_{k\bar{j}}\tilde{R}_{i\bar{l}}),\\
~\\
B_{i\bar{j}k\bar{l}}:=R_{i\bar{j}k\bar{l}}-S_{i\bar{j}k\bar{l}}^c
-P_{i\bar{j}k\bar{l}}^c.
\end{array}\right.
\end{eqnarray}

If $$(g^{i\bar{j}}):=\text{transpose of $(g_{i\bar{j}})^{-1}$},$$
then the pointwise squared norms of these tensors and forms are
defined as follows:
\begin{eqnarray}\label{localdecomp22}
\left\{ \begin{array}{ll}
|R^c|^2:=R_{i\bar{j}k\bar{l}}^cR_{p\bar{q}r\bar{s}}^c
g^{i\bar{q}}g^{p\bar{j}}g^{k\bar{s}}g^{r\bar{l}},\\
~\\
|\text{Ric}(\omega)|^2:= R_{i\bar{j}}R_{p\bar{q}}g^{i\bar{q}}g^{p\bar{j}},\\
~\\
\cdots.
\end{array}\right.
\end{eqnarray}

These pointwise squared norms satisfy (cf. \cite[Lemma 3.3]{Li2}):
\begin{lemma}\label{normrelation2}
If $(M,g)$ is a complex $n$-dimensional K\"{a}hler manifold, then
\begin{eqnarray}\label{normrelation22}
\left\{ \begin{array}{ll} |{\rm Ric}(\omega)|^2=|{\rm\tilde{R}ic}(\omega)|^2+
\frac{s^2_g}{4n},\\
~\\
|S^c|^2=\frac{s^2_g}{2n(n+1)},\\
~\\
|P^c|^2=\frac{4}{n+2}|{\rm\tilde{R}ic}(\omega)|^2.
\end{array}\right.
\end{eqnarray}
The K\"{a}hler metric $g$ is of constant holomorphic sectional curvature  (resp. Einstein) if and only if $S^c=P^c\equiv0$ (resp. $P^c\equiv0$). In particular, $|{\rm Ric}(\omega)|^2\geq
\frac{s^2_g}{4n}$ and with equality if and only if $g$ is Einstein. Moreover, if $(M,g)$ is compact, then
\be\label{normrelation33}\int_M|R^c|^2{\rm dvol}=
\int_M(|S^c|^2+|P^c|^2+|B|^2){\rm dvol}\ee
as the decomposition {\rm(\ref{decompostion2})} is orthogonal under the $L^2$-norm.
\end{lemma}

Now a natural question is, for a K\"{a}hler manifold $(M,g)$, what the relations are between the pointwise squared norms in Riemannian and K\"{a}hler cases. Indeed, they are related by the following proposition, which, to the author's best knowledge, \emph{never appears as explicitly as ours in previous literature}.

\begin{proposition}\label{relationbetweenrieka}
 Suppose $(M,J,g)$ is a complex $n$-dimensional K\"{a}hler manifold. Then the pointwise squared norms of  $\text{Ric}(g)$, $\text{Ric}(\omega)$, $R$ and $R^c$ are related by
 \be\label{relationbetweenrieka2}
 |\text{Ric}(g)|^2=2|\text{Ric}(\omega)|^2,\qquad |R|^2=4|R^c|^2.\ee
\end{proposition}

\begin{proof}
Since $(M,g)$ is K\"{a}hler, we can choose a (locally defined) frame field of the Riemannian manifold $(M,g)$ in such a manner: $\{e_i,e_{i+n}=Je_{i}~|~1\leq i\leq n\}$. Then
$$\big\{u_i:=\frac{1}{\sqrt{2}}(e_i-\sqrt{-1}Je_i)~|~1\leq i\leq n\big\}$$
is a $(1,0)$-type unitary field. Note that the Ricci tensor of a K\"{a}hler metric is $J$-invariant: ${\rm Ric}(Jx,Jy)={\rm Ric}(x,y)$. Then
\be\begin{split} R_{i\bar{j}}&={\rm Ric}(u_i,\bar{u_j})\\
&=\frac{1}{2}{\rm Ric}(e_i-\sqrt{-1}Je_i,e_j+\sqrt{-1}Je_j)\\
&={\rm Ric}(e_i,e_j)+\sqrt{-1}{\rm Ric}(e_i,Je_j)\qquad \big(\text{${\rm Ric}(Jx,Jy)={\rm Ric}(x,y)$}\big)\\
&=R_{ij}+\sqrt{-1}R_{i,j+n}.\end{split}\nonumber\ee

With respect to these two fields $\{e_i\}$ and $\{u_i\}$, we have $g_{ij}=g_{i\bar{j}}=\delta_{ij}$ and therefore
\be\begin{split}
|{\rm\tilde{R}ic}(\omega)|^2&=\sum_{1\leq i,j\leq n}R_{i\bar{j}}R_{j\bar{i}}\qquad (g_{i\bar{j}}=\delta_{ij})\\
&=\sum_{1\leq i,j\leq n}(R_{ij}+\sqrt{-1}R_{i,j+n})(R_{ji}+\sqrt{-1}R_{j,i+n})\\
&=\sum_{1\leq i,j\leq n}(R_{ij}^2+R^2_{i,j+n})\qquad (R_{ij}=R_{ji},R_{i,j+n}=-R_{j,i+n})\\
&=\sum_{1\leq i,j\leq n}\frac{1}{2}(R_{ij}^2+R^2_{i+n,j+n}+R_{i,j+n}^2+
R^2_{i+n,j})\qquad(R_{ij}=R_{i+n,j+n})\\
&=\frac{1}{2}\sum_{1\leq i,j\leq 2n}R_{ij}^2\\
&=\frac{1}{2}|\text{Ric}(g)|^2.\end{split}\nonumber\ee

Via various symmetric and anti-symmetric properties of the curvature tensor $R$ and its $J$-invariant property $R(Jx,Jy,z,w)=R(x,y,z,w)$ due to the K\"{a}hlerness of $g$, we can show that
\be\begin{split}
R_{i\bar{j}k\bar{l}}^c=&R(u_i,\bar{u_j},u_k,\bar{u_l})\\
=&
\frac{1}{4}R(e_i-\sqrt{-1}Je_i,
e_j+\sqrt{-1}Je_j,e_k-\sqrt{-1}Je_k,e_l+\sqrt{-1}Je_l)\\
=&\cdots\\
=&\Big[R(e_i,e_j,e_k,e_l)-R(e_i,Je_j,e_k,Je_l)\Big]+
\sqrt{-1}\Big[R(e_i,Je_j,e_k,e_l)+R(e_i,e_j,e_k,Je_l)\Big]\\
=&\big(R_{ijkl}-R_{i,j+n,k,l+n}\big)+\sqrt{-1}\big(
R_{i,j+n,k,l}+R_{i,j,k,l+n}\big)
\end{split}\nonumber\ee
and thus
\be\begin{split}\label{1/4}
|R^c|^2=&\sum_{1\leq i,j,k,l\leq n}R_{i\bar{j}k\bar{l}}^cR_{j\bar{i}l\bar{k}}^c
\qquad(g_{i\bar{j}}=\delta_{ij})\\
=&\sum_{1\leq i,j,k,l\leq n}\big[(R_{ijkl}-R_{i,j+n,k,l+n})^2+(R_{i,j+n,k,l}+R_{i,j,k,l+n})^2\big]\\
=&\sum_{1\leq i,j,k,l\leq n}(R_{ijkl}^2+R_{i,j+n,k,l+n}^2+R_{i,j+n,k,l}^2+R_{i,j,k,l+n}^2)\\
=&\frac{1}{4}\sum_{1\leq i,j,k,l\leq 2n}R_{ijkl}^2\\
=&\frac{1}{4}|R|^2.
\end{split}\ee
Here the third equality in (\ref{1/4}) is due to the facts that
$$R_{ijkl}R_{i,j+n,k,l+n}=-R_{jikl}R_{j,i+n,k,l+n},\qquad R_{i,j+n,k,l}R_{i,j,k,l+n}=-R_{j,i+n,k,l}R_{j,i,k,l+n}$$
and thus
$$\sum_{1\leq i,j,k,l\leq n}R_{ijkl}R_{i,j+n,k,l+n}=\sum_{1\leq i,j,k,l\leq n}R_{i,j+n,k,l}R_{i,j,k,l+n}=0.$$
\end{proof}

\subsection{Integral formulas and inequalities on compact K\"{a}hler manifolds}\label{subsection3.2}
The purpose of this subsection is to recall several classical integral formulas and inequalities on compact K\"{a}hler manifold relating the first Chern class $c_1$, the scalar curvature $s_g$, the K\"{a}hler form $\omega$ and the pointwise squared norm $|\text{Ric}(\omega)|^2$.

The following lemma relates $c_1$ and $\omega$ to $s_g$ and $|{\rm Ric}(\omega)|^2$.
\begin{lemma}
Suppose $(M,g)$ is a complex $n$-dimensional compact K\"{a}hler manifold.
Then we have
\be\label{c1}
\int_Mc_1(M)\wedge[\omega]^{n-1}=\frac{1}{2n}\int_M{s_g}\cdot\omega^n,\ee
and
\be\label{c1square}
\begin{split}
\int_Mc_1^2(M)\wedge[\omega]^{n-2}&~=\int_M
\big(\frac{s^2_g}{4}-|{\rm Ric}(\omega)|^2\big)
\cdot\frac{\omega^n}{n(n-1)}\\
&\stackrel{(\ref{normrelation22})}{=}\int_M
\big(\frac{n-1}{4n}s^2_g-|{\rm\tilde{R}ic}(\omega)|^2\big)
\cdot\frac{\omega^n}{n(n-1)}.\end{split}\ee
\end{lemma}

\begin{remark}~
\begin{enumerate}
\item
(\ref{c1}) is classical and the proof is easy (cf. \cite[p. 18]{Ti}). (\ref{c1square}) is essentially due to Apte in \cite{Ap}. We refer the reader to \cite[Lemma 3.1, Remark 3.2]{Li2} for more details about (\ref{c1square}).

\item
In different books/papers, the definition of K\"{a}hler form differs by a positive constant, which also cause a difference up to a positive constant in (\ref{c1}) and (\ref{c1square}). For more details see \cite[Remark 3.2]{Li2}.
\end{enumerate}
\end{remark}

We now end this preliminary section by the following lemma.
\begin{lemma}
Suppose $(M_i,g_i)$ $(i=1,2)$ are two compact Riemannian manifolds. If $s_{g_2}$ is constant and they satisfy
$$\text{Vol}(M,g_1)=\text{Vol}(M,g_2),\qquad\int_{M_1}s_{g_1}
{\rm dvol}=\int_{M_2}s_{g_2}
{\rm dvol}.$$
Then we have
\be\label{sclarinequality}\int_{M_1}(s_{g_1}^2-s_{g_2}^2)
{\rm dvol}\geq 0,\ee
with equality if and only if $s_{g_1}$ is constant.
Furthermore, if $(M_1, g_1)$ is K\"{a}hler and $c_1(M_1)\in\mathbb{R}[\omega_1]$, then the equality case of (\ref{sclarinequality}) holds if and only if $g_1$ is Einstein.
\end{lemma}

\begin{proof}
(\ref{sclarinequality}) is a direct corollary of Cauchy-Schwarz inequality:
$$\int_{M_1}s_{g_1}^2\text{dvol}\geq\frac{(\int_{M_1}s_{g_1}
\text{dvol})^2}{\text{Vol}(M,g_1)}=
\frac{(\int_{M_2}s_{g_2}
\text{dvol})^2}{\text{Vol}(M,g_2)}=\int_{M_2}s_{g_2}^2\text{dvol},$$
where the equality holds if and only if $s_{g_1}$ is constant. The second assertion is due to a well-known fact that a constant scalar curvature K\"{a}hler metric whose K\"{a}hler class is proportional to the first Chern class must be Einstein (cf. \cite[p. 19]{Ti}).
\end{proof}

\section{Proofs of Theorems \ref{main theorem} and \ref{maincorollary}}\label{section4}
With the preliminaries established in Section \ref{section3}, we are now ready to prove Theorem \ref{main theorem} as well as Theorem \ref{maincorollary}. Our main strategy is to show that, under the assumptions in Theorem \ref{main theorem}, the K\"{a}hler manifold in question has positive constant holomorphic sectional curvature. Then this K\"{a}hler manifold is holomorphically isometric to $(\mathbb{C}P^n,J_0,g_0)$.

We first assume that $(M,g)$ be an $m$-dimensional compact Riemannian manifold, not necessarily K\"{a}hler, and $\{\lambda_{k,p}\}$ are the spectra of the Laplacian of $g$ mentioned at the beginning of this paper. Then for any positive integer $N$ we have the following Minakshisundaram asymptotic expansion formula (cf. \cite[\S 4]{Pa2}), which is the integration over the asymptotic expansion of the heat kernel for Laplacian:
\be\label{mpgformula}
\sum_{k=0}^{\infty}\exp(-\lambda_{k,p}t)=\frac{1}{(4\pi t)^{\frac{m}{2}}}
\sum_{i=0}^{N}a_{i,p}t^i+O(t^{N-\frac{m}{2}+1}),\qquad t\downarrow0.\ee

Minakshisundaram's coefficients $a_{0,p}$£¬ $a_{1,p}$ and $a_{2,p}$ in (\ref{mpgformula}) were explicitly determined by Patodi in \cite[Prop. 2.1]{Pa}:
\begin{eqnarray}\label{patodiformula}
\left\{ \begin{array}{ll} a_{0,p}={m\choose p}\text{vol}(M,g),\\
~\\
a_{1,p}=[\frac{1}{6}{m\choose p}-{m-2\choose p-1}]\int_{M}s_g\text{dvol},\\
~\\
a_{2,p}=\int_{M}(\lambda_1|\text{R}|^2+\lambda_2|
\text{Ric}(g)|^2+\lambda_3s_g^2)\text{dvol},
\end{array}\right.
\end{eqnarray}
where
\begin{eqnarray}\label{patodicoefficient}
\left\{ \begin{array}{ll} \lambda_1=\frac{1}{180}{m\choose p}
-\frac{1}{12}{m-2\choose p-1}+\frac{1}{2}{m-4\choose p-2},\\
~\\
\lambda_2=-\frac{1}{180}{m\choose p}
+\frac{1}{2}{m-2\choose p-1}-2{m-4\choose p-2},\\
~\\
\lambda_3=\frac{1}{72}{m\choose p}
-\frac{1}{6}{m-2\choose p-1}+\frac{1}{2}{m-4\choose p-2}.
\end{array}\right.
\end{eqnarray}

Note that the coefficient in front of $a_{1,p}$ in (\ref{patodiformula}) is
\be\label{a1pcoefficient}\frac{1}{6}{m\choose p}-{m-2\choose p-1}
=
\frac{(m-2)!}{p!(m-p)!}\big[p^2-mp+\frac{m(m-1)}{6}\big],\ee
which is nonzero exactly under the requirement (\ref{relation}) when $m=2n$.
The following lemma summarizes how much geometric information of a compact Riemannian manifold can be reflected by the spectrum set $\text{Spec}^p(\cdot)$ for a fixed $p$.

\begin{lemma}
If two compact Riemannian manifolds $(M_i,g_i)$ $(i=1,2)$ $(\text{dim}(M_i)\geq 3)$ satisfy $$\text{Spec}^p(M_1,g_1)=\text{Spec}^p(M_2,g_2)$$
for a fixed $p$, then we have
\begin{eqnarray}\label{lemma1}
\left\{ \begin{array}{ll}
{\rm dim}(M_1)={\rm dim}(M_2)=:m\\
~\\
{\rm Vol}(M_1,g_1)={\rm Vol}(M_2,g_2)\\
~\\
\int_{M_1}s_{g_1}{\rm dvol}=
\int_{M_2}s_{g_2}{\rm dvol},\qquad \text{{\rm provided that} $p^2-mp+\frac{m(m-1)}{6}\neq0$}.
\end{array}\right.
\end{eqnarray}

If moreover $g_i$ are K\"{a}hler and $m=2n$, then
\be\label{lemma3}\begin{split}
&a_{2,p}(M_i)\\
=&\int_{M_i}
\Big[\big(\frac{2}{n(n+1)}\lambda_1+
\frac{1}{2n}\lambda_2+\lambda_3\big)s_{g_i}^2+\big(\frac{16}{n+2}\lambda_1
+2\lambda_2\big)|
{\rm\tilde{R}ic}(\omega_i)|^2+
4\lambda_1|B(g_i)|^2\Big]{\rm dvol}.\end{split}\ee
for $i=1,2$ and thus they are equal, where $B(g_i)$ denotes the Bochner curvature tensor of $g_i$.
\end{lemma}

\begin{remark}
(\ref{lemma3}) was attempted to derive in \cite[(2,5)]{GG} but the coefficients obtained in \cite[(2,5)]{GG} were \emph{false}, where our $\lambda_i$ and $2n$ were denoted by $C_i$ and $m$ respectively. We shall see later that the precise values of these coefficients are crucial in establishing Theorem \ref{main theorem}.
\end{remark}

\begin{proof}
The first equality in (\ref{lemma1}) follows from (\ref{mpgformula}), which, together with the expression of $a_{0,p}$ in (\ref{patodiformula}), implies the second equality in (\ref{lemma1}). The expression for $a_{1,p}$ in (\ref{patodicoefficient}) and (\ref{a1pcoefficient}) lead to the third one in (\ref{lemma1}). (\ref{lemma3}) follows from Lemma \ref{normrelation2} and Proposition \ref{relationbetweenrieka} as follows:
\be\begin{split}
&a_{2,p}(M_i)\\
=&\int_{M_i}
\big(\lambda_1|R|^2+\lambda_2|
{\rm Ric}(g_i)|^2+\lambda_3s_{g_i}^2\big)\text{dvol}
\qquad\big((\ref{patodiformula})\big)\\
=&
\int_{M_i}\big[4\lambda_1|R^c|^2+
2\lambda_2|{\rm Ric}(\omega_i)|^2+
\lambda_3s_{g_i}^2\big]\text{dvol}
\qquad\big((\ref{relationbetweenrieka2})\big)\\
=&\int_{M_i}\Big[4\lambda_1(|S^c|^2+|P^c|^2+|B|^2)+
2\lambda_2|{\rm Ric}(\omega_i)|^2
+
\lambda_3s_{g_i}^2\Big]\text{dvol}
\qquad\big((\ref{normrelation33})\big)\\
=&\int_{M_i}
\Big[4\lambda_1(\frac{s^2_{g_i}}{2n(n+1)}+
\frac{4}{n+2}|{\rm\tilde{R}ic}(\omega_i)|^2+|B|^2)+
2\lambda_2(|{\rm\tilde{R}ic}(\omega_i)|^2+\frac{s_{g_i}^2}{4n})+
\lambda_3s_{g_i}^2\Big]\text{dvol}~\big((\ref{normrelation22})\big)\\
=&\int_{M_i}
\Big[\big(\frac{2}{n(n+1)}\lambda_1+
\frac{1}{2n}\lambda_2+\lambda_3\big)s_{g_i}^2+\big(\frac{16}{n+2}\lambda_1
+2\lambda_2\big)|{\rm\tilde{R}ic}(\omega_i)|^2+4\lambda_1|B|^2\Big]\text{dvol}.
\end{split}\nonumber\ee
\end{proof}

The following lemma tells us that the second Betti number of the manifold in question in Theorem \ref{main theorem} is equal to $1$.
\begin{lemma}\label{bettinumberlemma}
Suppose $(M,g)$ is a complex $n$-dimensional compact K\"{a}hler manifold with $\text{Spec}^p(M,g)=\text{Spec}^p(\mathbb{C}P^n, g_0)$ for $p$ even and $2\leq p\leq 2(n-1)$. Then the Betti numbers $b_{2i}(M)=1$ for $1\leq i\leq\frac{p}{2}$. In particular, the second Betti number $b_2(M)=1$.
\end{lemma}
\begin{proof}
First let us recall the following facts on the even-th Betti numbers $b_{2i}$ of \emph{general} compact K\"{a}hler manifolds:
\be\label{betti}
1\leq b_{2}\leq b_4\leq\cdots\leq b_{2[\frac{n}{2}]}=b_{2(n-[\frac{n}{2}])}\geq b_{2(n-[\frac{n}{2}]+1)}\geq\cdots\geq b_{2(n-1)}\geq 1.
\ee

Indeed, the K\"{a}hler class $[\omega]$ of a compact K\"{a}hler manifold $M$ represents a nonzero real $2$-dimensional cohomology element and thus $b_2\geq 1$. The hard Lefschetz theorem (\cite[p. 122]{GH}) tells us that
$$[\omega]^{n-2i}\wedge(\cdot):~H^{2i}(M;\mathbb{R})\rightarrow
H^{2(n-i)}(M;\mathbb{R}),\qquad 1\leq i\leq[\frac{n}{2}]-1,$$ is an isomorphism. This means that the map
$$[\omega]\wedge(\cdot):~H^{2i}(M;\mathbb{R})\rightarrow
H^{2(i+1)}(M;\mathbb{R}),\qquad 1\leq i\leq[\frac{n}{2}]-1,$$
is injective and thus $b_{2i}\leq b_{2(i+1)}$ for $1\leq i\leq[\frac{n}{2}]-1$. The second part in (\ref{betti}) is due to the Poincar\'{e} duality.

Note that the multiplicity of $0$ in $\text{Spec}^p(M,g)$ is exactly $b_p(M)$. Thus the condition in Lemma \ref{bettinumberlemma} and the Poincar\'{e} duality imply that $b_p(M)=b_{2n-p}(M)=1$, which, together with (\ref{betti}), yield Lemma \ref{bettinumberlemma}.
\end{proof}

We assume from now on the conditions made in Theorem \ref{main theorem} and proceed to prove it.

First we show that $M$ is Fano. Indeed, Lemma \ref{bettinumberlemma} says that $b_2(M)=1$ and thus $c_1(M)=\lambda[\omega]$ with $\lambda\in\mathbb{R}$. Then
\be\begin{split}
\lambda=\frac{\int_Mc_1(M)
\wedge[\omega]^{n-1}}{\int_M\omega^n}
&=\frac{\int_Ms_g\cdot\omega^{n}}{2n\int_M\omega^n}\qquad\big((\ref{c1})\big)\\
&=\frac{\int_Ms_g\text{dvol}}{2n\int_M\text{dvol}}\qquad\big((\ref{volumeelement})\big)\\
&=\frac{\int_{\mathbb{C}P^n}s_{g_0}\text{dvol}}
{2n\int_{\mathbb{C}P^n}\text{dvol}}\qquad\big((\ref{lemma1})\big)\\
&=\frac{s_{g_0}}{2n}>0
\end{split}\nonumber\ee
and therefore $M$ is Fano.

Next we shall show that the K\"{a}hler metric $g$ has constant holomorphic sectional curvature, which will be derived from the following lemma.

\begin{lemma}\label{keylemma}
Under the conditions assumed in Theorem \ref{main theorem}, $M$ satisfies the following integral formula:
\be\label{keyintegralformula}
\big(\frac{4n+2}{(n+1)(n+2)}\lambda_1+
\frac{1}{2}\lambda_2+\lambda_3\big)\int_M(s_g^2-s_{g_0}^2)
{\rm dvol}+4\lambda_1\int_M|B|^2{\rm dvol}=0.
\ee
\end{lemma}
\begin{proof}
Note that the Fubini-Study metric $g_0$ of $(\mathbb{C}P^n,J_0,g_0)$ has (positive) constant holomorphic sectional curvature and thus the two tensors $P^c=B\equiv0$ for $g_0$. This, together with (\ref{lemma3}), implies that
\be
\begin{split}&\int_{M}
\Big[\big(\frac{2}{n(n+1)}\lambda_1+
\frac{1}{2n}\lambda_2+\lambda_3\big)s_{g}^2+\big(\frac{16}{n+2}\lambda_1
+2\lambda_2\big)|{\rm\tilde{R}ic}(\omega)|^2+4\lambda_1|B|^2\Big]\text{dvol}\\
=&\int_{\mathbb{C}P^n}\big(\frac{2}{n(n+1)}\lambda_1+
\frac{1}{2n}\lambda_2+\lambda_3\big)s_{g_0}^2\text{dvol}\\
=&\int_{M}\big(\frac{2}{n(n+1)}\lambda_1+
\frac{1}{2n}\lambda_2+\lambda_3\big)s_{g_0}^2\text{dvol},\qquad\big(\text{$s_{g_0}$ constant, Vol$(M,g)=$Vol$(\mathbb{C}P^n,g_0)$}\big)
\end{split}\nonumber\ee
which yields
\be\label{prove1}
\int_{M}
\Big[\big(\frac{2}{n(n+1)}\lambda_1+
\frac{1}{2n}\lambda_2+\lambda_3\big)\big(s_{g}^2-s_{g_0}^2\big)+\big(\frac{16}{n+2}\lambda_1
+2\lambda_2\big)|{\rm\tilde{R}ic}(\omega)|^2+4\lambda_1|B|^2\Big]\text{dvol}=0.
\ee

On the other hand, with the fact (\ref{volumeelement}) that $\omega^n$ differ from the volume element by a universal constant in mind, we have
\be\begin{split}
&\int_M
\big(\frac{n-1}{4n}s^2_g-|{\rm\tilde{R}ic}(\omega)|^2\big)
\cdot\frac{\omega^n}{n(n-1)}\\
=&
\big(\int_Mc_1^2(M)\wedge[\omega]^{n-2}\big)\qquad\big((\ref{c1square})\big)\\
=&\frac{\big(\int_Mc_1(M)\wedge[\omega]^{n-1}\big)^2}
{\int_M\omega^n}\qquad\big(\text{$c_1(M)\in\mathbb{R}\omega$}\big)\\
=&\frac{\big(\int_Ms_g\omega^{n}\big)^2}
{4n^2\int_M\omega^n}\qquad\big((\ref{c1})\big)\\
=&\frac{\big(\int_{\mathbb{C}P^n}s_{g_0}\omega_0^{n}\big)^2}
{4n^2\int_{\mathbb{C}P^n}\omega_0^n}\qquad\big((\ref{lemma1})\big)\\
=&\frac{s^2_{g_0}}{4n^2}\int_{\mathbb{C}P^n}\omega_0^n\qquad\big(\text{$s_{g_0}$ constant}\big)\\
=&\frac{s^2_{g_0}}{4n^2}\int_{M}\omega^n,
\end{split}\nonumber\ee
which leads to
\be\label{prove1.5}\frac{1}{n(n-1)}\int_M
\big(\frac{n-1}{4n}s^2_g-|{\rm\tilde{R}ic}(\omega)|^2\big)
\text{dvol}=\frac{s^2_{g_0}}{4n^2}\int_{M}\text{dvol}.\ee

Rewriting (\ref{prove1.5}) by singling out the term $|{\rm\tilde{R}ic}(\omega)|^2$ we have
\be\label{prove2}\int_M|{\rm\tilde{R}ic}(\omega)|^2\text{dvol}=
\frac{n-1}{4n}\int_M(s_g^2-s_{g_0}^2)\text{dvol}.\ee
Our (\ref{keyintegralformula}) follows now from substituting (\ref{prove2}) into (\ref{prove1}).
\end{proof}

Theorems \ref{main theorem} and \ref{maincorollary} now follow easily from the following Propositions \ref{technicallemma1} and \ref{lemmarelation} respectively, whose proofs are a little bit complicated and thus will be postponed to Section \ref{section5}.

\begin{proposition}\label{technicallemma1}
Suppose $p$ is even, $2\leq p\leq m-2=2(n-1)$ and $m\geq 4$. Then
\be\label{technicalinequality1}\frac{4n+2}{(n+1)(n+2)}\lambda_1+
\frac{1}{2}\lambda_2+\lambda_3>0\ee
and
\be\label{technicalinequality2}\lambda_1\geq0,\ee
where the equality case in {\rm(\ref{technicalinequality2}) } holds if and only if $(p,m)=(2,16)$.
\end{proposition}

We can now prove Theorem \ref{main theorem}.
\begin{proof}
If $(p,m)\neq(2,16)$, we know from (\ref{technicalinequality1}), (\ref{technicalinequality2}), (\ref{keyintegralformula}) and (\ref{sclarinequality}) that
\be\int_M(s_g^2-s_{g_0}^2)\text{dvol}=\int_M|B|^2\text{dvol}=0.\ee

This implies that $g$ is Einstein \big(see the sentence after (\ref{sclarinequality})\big), which is equivalent to the tensor $P^c\equiv0$, and the Bochner tensor $B\equiv0$. This means $g$ has constant holomorphic sectional curvature, which is positive in our case as $c_1(M)>0$ established before Lemma \ref{keylemma}. Then our conclusion follows from the uniformization theorem for positive constant holomorphic sectional curvature metrics of compact K\"{a}hler manifolds.

If $(p,m)=(2,16)$, then $\lambda_1=0$ and in this case we can \emph{only} conclude that $g$ is Einstein. But in this case our conclusion still holds due to Theorem \ref{theorem2fujita}.
\end{proof}

\begin{remark}
Of course in the above proof we can avoid distinguishing the two cases $(p,m)\neq$ or $=(2,16)$ by both resorting to
Theorem \ref{theorem2fujita}. But the process of the proof illustrates that, except the case $(p,m)=(2,16)$, our Theorem \ref{main theorem} is \emph{independent} of the very recent Theorem \ref{theorem2fujita} and can be deduced from some classical (but important!) complex geometry tools. Nevertheless, the proof for the case $(p,m)=(2,16)$ relies essentially on Theorem \ref{theorem2fujita}, which is exactly where the mistake occurs in \cite{Go}.
\end{remark}

Theorem \ref{maincorollary} follows from the following
\begin{proposition}\label{lemmarelation}
The positive integer solutions to the equation
\be\label{eq}p^2-2np+\frac{n(2n-1)}{3}=0\ee
such that $p$ is \emph{even}
are precisely of the forms
\be\text{$(n,p)=(n_k,p_k)$ or $(n_k,2n_k-p_k)$}a\quad (k=1,2,\cdots)\nonumber\ee
with $(n_k,p_k)$ satisfying the following recursive formula
\begin{eqnarray}\label{recursiveformula}
\left\{ \begin{array}{ll} n_{k+1}=265n_k-168p_k+48\\
~\\
p_{k+1}=112n_k-71p_k+20\\
~\\
(n_1,p_1)=(48,20).
\end{array} \right.
\end{eqnarray}
\end{proposition}

\section{Appendix}\label{section5}
\subsection{Proof of Proposition \ref{technicallemma1}}
Recall the definition of $\lambda_i$ in (\ref{patodicoefficient}) and that $m=2n$, and direct calculations show that
\be\begin{split}\label{positivity}
&\frac{4n+2}{(n+1)(n+2)}\lambda_1+
\frac{1}{2}\lambda_2+\lambda_3\\
=&\frac{m^2+10m+12}{90(m+2)(m+4)}{m\choose p}
+\frac{m(m-2)}{12(m+2)(m+4)}{m-2\choose p-1}-\frac{m(m-2)}{2(m+2)(m+4)}{m-4\choose p-2}.\end{split}\ee

In order to prove the positivity of (\ref{positivity}) and the nonnegativity of $\lambda_1$ under our restrictions that $p$ and $m$ be even, $2\leq p\leq m-2$ and $m\geq4$, we would like to first investigate the following general linear combination of ${m\choose p}$, ${m-2\choose p-1}$ and ${m-4\choose p-2}$:
\be\begin{split}
&\alpha{m\choose p}
+\beta{m-2\choose p-1}+\gamma{m-4\choose p-2}\qquad(\alpha,\beta,\gamma\in\mathbb{R})\\
=&\alpha\frac{m!}{p!(m-p)!}+\beta\frac{(m-2)!}{(p-1)!(m-p-1)!}
+\gamma\frac{(m-4)!}{(p-2)!(m-p-2)!}\\
=&\frac{(m-4)!}{p!(m-p)!}\Big[\alpha m(m-1)(m-2)(m-3)\\
&+
\beta(m-2)(m-3)p(m-p)+\gamma p(p-1)(m-p)(m-p-1)\Big]\\
=:&\frac{(m-4)!}{p!(m-p)!}f(p,m,\alpha,\beta,\gamma)\\
=:&\frac{(m-4)!}{p!(m-p)!}\Big[\alpha m(m-1)(m-2)(m-3)+
g(p,m,\beta,\gamma)\Big].
\end{split}\nonumber\ee

Note that $f$ and $g$ satisfy
\begin{eqnarray}\label{symmetry}
\left\{ \begin{array}{ll} f(p,m,\alpha,\beta,\gamma)=f(m-p,m,\alpha,\beta,\gamma)\\
~\\
g(p,m,\beta,\gamma)=g(m-p,m,\beta,\gamma).
\end{array} \right.
\end{eqnarray}

With this formulation it suffices to show under our restrictions that
\be\label{firstinequality}f(p,m,\frac{m^2+10m+12}{90(m+2)(m+4)},
\frac{m(m-2)}{12(m+2)(m+4)},
-\frac{m(m-2)}{2(m+2)(m+4)})>0\ee
and
$$f(p,m,\frac{1}{180},-\frac{1}{2},\frac{1}{2})\geq0$$
 with equality if and only if $(p,m)=(2,16)$.

We now temporarily fix $m,\beta,\gamma$, view $g(p,m,\beta,\gamma)=:g(p)$ as a function with \emph{real} variable $p\in[2,m-2]$ and investigate the minimal value of $g(p)$ in the interval $[2,m-2]$. First notice that $g(p)$ satisfies the following
\begin{lemma}\label{lemmatechnical}
The function $g'(p)$ has three roots
$p_1=\frac{m}{2}=n$, $p_2$ and $p_3$, where $p_{2,3}$ satisfy $p_2+p_3=m$ and thus $g(p_2)=g(p_3)=:g(p_{2,3})$ by (\ref{symmetry}). Moreover, we have
\begin{eqnarray}\label{criticalpointvalue}
\left\{ \begin{array}{ll}
g''(p_1)=
(-2\beta-\gamma)m^2+(10\beta+2\gamma)m+(-12\beta-2\gamma)\\
~\\
g''(p_{2,3})=2\gamma(2p_{2,3}-m)^2
\end{array} \right.
\end{eqnarray}
and
\be\label{gp23}g(p_{2,3})
=-\frac{1}{4\gamma}\big[\beta(m-2)(m-3)-
\gamma(m-1)\big]^2\qquad\text{provided $\gamma\neq0$}.\ee
\end{lemma}
\begin{proof}
\be\label{first0}\begin{split}
g(p)&=\beta(m-2)(m-3)p(m-p)+\gamma p(p-1)(m-p)(m-p-1)\\
&=p(p-m)\big[\gamma(p-1)(p-m+1)-\beta(m-2)(m-3)\big]\\
&=(p^2-mp)\big[\gamma p^2-\gamma mp+\gamma(m-1)-\beta(m-2)(m-3)\big]
\end{split}\ee
and thus
\be\label{first}\begin{split}
g'(p)&=(2p-m)\big[\gamma p^2-\gamma mp+\gamma(m-1)-\beta(m-2)(m-3)\big]+(p^2-mp)(2\gamma p-\gamma m)\\
&=(2p-m)\big[2\gamma p^2-2\gamma mp+\gamma(m-1)-\beta(m-2)(m-3)\big]\\
&=:(2p-m)h(p).
\end{split}
\ee

This means that $g'(p)$ has three roots
$p_1=\frac{m}{2}=n$, $p_2$ and $p_3$, where $p_{2,3}$ satisfy $h(p_{2,3})=0$ and particularly $p_2+p_3=m$. Also note that
\be g''(p)=2h(p)+2\gamma(2p-m)^2,\nonumber\ee
which implies that
 $$g''(p_1)=g''(\frac{m}{2})=2h(\frac{m}{2})=\cdots=
(-2\beta-\gamma)m^2+(10\beta+2\gamma)m+(-12\beta-2\gamma)$$
and
$$g''(p_{2,3})=2\gamma(2p_{2,3}-m)^2,$$
and thus completes the proof of (\ref{criticalpointvalue}).
Next we calculate $g(p_{2,3})$. Note that $p_{2,3}$ satisfy $h(p_{2,3})=0$ \big(see (\ref{first})\big), which, when $\gamma\neq0$, is equivalent to
\be\label{first1.5} p^2_{2,3}-mp_{2,3}=\frac{\beta}{2\gamma}(m-2)(m-3)
-\frac{1}{2}(m-1),\qquad(\gamma\neq0).\ee

Substituting (\ref{first1.5}) into (\ref{first0}) we have
\be\begin{split}
g(p_{2,3})&=\big[\frac{\beta}{2\gamma}(m-2)(m-3)-\frac{1}{2}(m-1)\big]
\big[-\frac{\beta}{2}(m-2)(m-3)+\frac{\gamma}{2}(m-1)\big]\\
&=-\frac{1}{4\gamma}\big[\beta(m-2)(m-3)-\gamma(m-1)\big]^2,~(\gamma\neq0)
\end{split}\nonumber\ee
which completes the proof of (\ref{gp23}).
\end{proof}

We now discuss two cases respectively.

$\mathbf{Case~1}$. $$(\alpha,\beta,\gamma)=\big(\frac{m^2+10m+12}{90(m+2)(m+4)},
\frac{m(m-2)}{12(m+2)(m+4)},-\frac{m(m-2)}{2(m+2)(m+4)}\big).$$

In this case via (\ref{criticalpointvalue}) and direct calculations we have
$$g''(p_1)=\frac{m^2(m-2)(2m-1)}{6(m+2)(m+4)}>0,\qquad g''(p_{2,3})<0.$$

This means that $p_1=\frac{m}{2}$ is the unique local minimum of $g(p)$ in $(2,m-2)$ and thus in this case
\be\begin{split}\label{minimum1}
&\min_{p\in[2,m-2]} f\big(p,m,\frac{m^2+10m+12}{90(m+2)(m+4)},
\frac{m(m-2)}{12(m+2)(m+4)},-\frac{m(m-2)}{2(m+2)(m+4)}\big)\\
=&\min\big\{f\big|_{p=\frac{m}{2}},f\big|_{p=2},f\big|_{p=m-2}\big\}.\end{split}\ee

Notice that
\be\label{first1}\begin{split}
&f\big|_{p=\frac{m}{2}}\\
=&
\frac{m^2+10m+12}{90(m+2)(m+4)}\cdot m(m-1)(m-2)(m-3)\\
&+\frac{m(m-2)}{12(m+2)(m+4)}\cdot(m-2)(m-3)\frac{m^2}{4}
-\frac{m(m-2)}{2(m+2)(m+4)}\cdot\frac{m^2}{4}\frac{(m-2)^2}{4}\\
=&\cdots\\
=&\frac{m(m-2)}{1440(m+2)(m+4)}(m^4+126m^3-400m^2-288m+576)\\
>&0\qquad(m\geq4)
\end{split}\ee
and
\be\label{first2}\begin{split}
&f\big|_{p=2}=f\big|_{p=m-2}\\
=&
\frac{m^2+10m+12}{90(m+2)(m+4)}\cdot m(m-1)(m-2)(m-3)\\
&+\frac{m(m-2)}{12(m+2)(m+4)}\cdot2(m-2)^2(m-3)
-\frac{m(m-2)}{2(m+2)(m+4)}\cdot2(m-2)(m-3)\\
=&\cdots\\
=&\frac{m(m-2)(m-3)}{90(m+2)(m+4)}(m^3+24m^2-148m+228)\\
>&0.\qquad(m\geq4)
\end{split}\ee

(\ref{minimum1}), (\ref{first1}) and (\ref{first2}) tell us that
$$\min_{p\in[2,m-2]} f\big(p,m,\frac{m^2+10m+12}{90(m+2)(m+4)},
\frac{m(m-2)}{12(m+2)(m+4)},-\frac{m(m-2)}{2(m+2)(m+4)}\big)>0$$
and thus
(\ref{firstinequality}) holds, which completes the first part of Proposition \ref{technicallemma1}.

$\mathbf{Case~2}.$
$$(\alpha,\beta,\gamma)=(\frac{1}{180},
-\frac{1}{12},\frac{1}{2}).$$

In this case via (\ref{criticalpointvalue}) we have
$$g''(p_1)=-\frac{1}{3}m^2+\frac{1}{6}m<0,\qquad g''(p_{2,3})>0.$$

This means that $p_{2,3}$ are the local minima of $g(p)$, which, together with the above-mentioned fact that $f\big|_{p=2}=f\big|_{p=m-2},$ implies
\begin{eqnarray}\label{first3}
\min_{p\in[2,m-2]} f(p,m,\frac{1}{180},
-\frac{1}{12},\frac{1}{2})=\left\{ \begin{array}{ll} \min\big\{f\big|_{p=2},f\big|_{p=p_{2,3}}\big\},&\text{if $p_{2,3}\in[2,m-2]$}\\
~\\
f\big|_{p=2}.&\text{if $p_{2,3}\not\in[2,m-2]$}
\end{array}\right.
\end{eqnarray}

Note that in this case
\be\label{first4}\begin{split}f\big|_{p=2}&=\frac{1}{180}
m(m-1)(m-2)(m-3)-\frac{1}{6}(m-2)^2(m-3)+
(m-2)(m-3)\\
&=\cdots\\
&=\frac{1}{180}(m-2)(m-3)(m-15)(m-16)\\
&\geq0,\end{split}\ee
with equality if and only if $m=16$ ($m=2n\geq4$ is even), and
\be\label{first5}\begin{split}f\big|_{p=p_{2,3}}&=\alpha m(m-1)(m-2)(m-3)+g(p_{2,3})\\
&\stackrel{(\ref{gp23})}{=}\alpha m(m-1)(m-2)(m-3)-\frac{1}{4\gamma}\big[\beta(m-2)(m-3)
-\gamma(m-1)\big]^2\\
&=\cdots\\
&=\frac{m}{1440}(3m^3-58m^2+83m-48)\qquad\big((\alpha,\beta,\gamma)=(\frac{1}{180},
-\frac{1}{12},\frac{1}{2})\big)\\
&>0.\qquad(m\geq18)
\end{split}\ee

Nevertheless, we can directly check from (\ref{first1.5}) and $(\beta,\gamma)=(-\frac{1}{12},\frac{1}{2})$ that
\be\label{first6}
p_{2,3}=\frac{1}{2}m\pm\sqrt{\frac{1}{6}m^2-
\frac{1}{12}m}\not\in[2,m-2]\qquad(4\leq m\leq16).\ee

Combining (\ref{first3}), (\ref{first4}), (\ref{first5}) and (\ref{first6}) we conclude, under the restrictions made in Proposition \ref{technicallemma1}, that
\[\lambda_1=f(p,m,\frac{1}{180},
-\frac{1}{12},\frac{1}{2})\geq 0\]
and with equality if and only if $(p,m)=(2,16)$. This completes the second part of Proposition \ref{technicallemma1}.

\subsection{Proof of Proposition \ref{lemmarelation}}
We know from (\ref{eq}) that
$$p=n\pm\sqrt{\frac{n(n+1)}{3}},$$
which implies that
\be\label{pell1}n(n+1)=3r^2\nonumber\ee with positive integer $r$, which is equivalent to
\be\label{pell2}(2n+1)^2-12r^2=1.\ee

(\ref{pell2}) is the famous \emph{Pell equation} (cf. \cite[\S XII]{Di}) whose positive integer solutions, denoted by $(\tilde{n}_k,\tilde{r}_k)$, are exactly parametrized by positive integers $k$ and satisfy the following recursive formula: \be\label{eq2}2\tilde{n}_k+1+\tilde{r}_k\sqrt{12}
=(7+2\sqrt{12})^k\qquad (k=1,2,\ldots).\ee

Thus expanding (\ref{eq2}) we obtain
\begin{eqnarray}\label{recursiveformula2}
\left\{ \begin{array}{ll} \tilde{n}_{k+1}=7\tilde{n}_k+12\tilde{r}_k+3\\
~\\
\tilde{r}_{k+1}=4\tilde{n}_k+7\tilde{r}_k+2\\
~\\
(\tilde{n}_1,\tilde{r}_1)=(3,2).
\end{array} \right.
\end{eqnarray}

This means, if we denote by
$\tilde{p}_k:=\tilde{n}_k-\tilde{r}_k$, then the positive solutions to (\ref{eq}) are precisely of the forms $(n,p)=(\tilde{n}_k,\tilde{p}_k)$ or $(\tilde{n}_k,2\tilde{n}_k-\tilde{p}_k)$. Therefore substituting $\tilde{p}_k=\tilde{n}_k-\tilde{r}_k$ into (\ref{recursiveformula2}) we get the following recursive formula for all the positive solutions to (\ref{eq}):
\begin{eqnarray}\label{recursiveformula3}
\left\{ \begin{array}{ll}
\text{$(n,p)=(\tilde{n}_k,\tilde{p}_k)$ or $(\tilde{n}_k,2\tilde{n}_k-\tilde{p}_k)$}\\
~\\
 \tilde{n}_{k+1}=19\tilde{n}_k-12\tilde{p}_k+3\\
~\\
\tilde{p}_{k+1}=8\tilde{n}_k-5\tilde{p}_k+1\\
~\\
(\tilde{n}_1,\tilde{p}_1)=(3,1).
\end{array} \right.
\end{eqnarray}

Note that $\tilde{p}_1=1$ is odd and so we deduce from (\ref{recursiveformula3}) that $\tilde{p}_{2k}$ (resp. $\tilde{p}_{2k-1}$) ($k\geq1$) are even (resp. odd). This means, if we denote by $(n_k,p_k):=(\tilde{n}_{2k},\tilde{p}_{2k})$, then all the positive integer solutions $(n,p)$ to (\ref{eq}) such that $p$ are \emph{even} are of the forms $(n,p)=(n_k,p_k)$ or $(n_k,2n_k-p_k)$. Applying iteration process (\ref{recursiveformula3}) twice yields (\ref{recursiveformula}) and thus completes the proof of Proposition \ref{lemmarelation}.

\end{document}